\documentclass[a4paper,11pt]{amsart}
\usepackage{hyperref,fancyhdr,mathrsfs,amsmath,amscd,amsthm,amsfonts,latexsym,amssymb}
\usepackage[all]{xy}

\DeclareMathOperator{\grad}{grad}

\DeclareMathOperator{\RicM}{^{\it M}\!Ric}
\DeclareMathOperator{\RicN}{^{\it N}\!Ric}

\DeclareMathOperator{\dif}{d}

\DeclareMathOperator{\Lie}{\mathcal{L}}

\renewcommand{\H}{\mathscr{H}}
\newcommand{\V}{\mathscr{V}}

\DeclareMathOperator{\dH}{d^{\H}}

\def \a{\alpha}

\def \b{\beta}

\def \g{\gamma}

\def \l{\lambda}

\def \O{\Omega}
\def \phi{\varphi}
\def \Phi{\varPhi}

\def \r{\rho}

\def \D{\Delta}
\def \R{\mathbb{R}}

\def \C{\mathbb{C}\,}

\def\widecheckg{g^{\hspace*{-2.5pt}\vbox to 5pt{\hbox to
0pt{\LARGE$\check{}$}}}\hspace*{2pt}}

\def\widecheckl{\lambda^{\hspace*{-3.5pt}\vbox to 8pt{\hbox to
0pt{\LARGE$\check{}$}}}\hspace*{2pt}}

\begin{document}

\title{On Ricci solitons and\\
twistorial harmonic morphisms}
\author{Paul Baird and Radu~Pantilie\,\dag}
\thanks{\dag\,Acknowledges that this work was done during a three months visit at the University of Brest, France, 
supported by the Centre National de la Recherche Scientifique, and by a grant of the Romanian National Authority for
Scientific Research, CNCS-UEFISCDI, project number PN-II-ID-PCE-2011-3-036}
\email{\href{mailto:Paul.Baird@univ-brest.fr}{Paul.Baird@univ-brest.fr},
       \href{mailto:Radu.Pantilie@imar.ro}{Radu.Pantilie@imar.ro}}
\address{P.~Baird, D\'epartement de Math\'ematiques, Laboratoire C.N.R.S. U.M.R. 6205,
Universit\'e de Bretagne Occidentale, 6, Avenue Victor Le Gorgeu, CS 93837,
29238 Brest Cedex 3, France}
\address{R.~Pantilie, Institutul de Matematic\u a ``Simion Stoilow'' al Academiei Rom\^ane,
C.P. 1-764, 014700, Bucure\c sti, Rom\^ania}

\newtheorem{thm}{Theorem}[section]
\newtheorem{lem}[thm]{Lemma}
\newtheorem{cor}[thm]{Corollary}
\newtheorem{prop}[thm]{Proposition}

\theoremstyle{definition}

\newtheorem{defn}[thm]{Definition}
\newtheorem{rem}[thm]{Remark}
\newtheorem{exm}[thm]{Example}

\numberwithin{equation}{section}

\maketitle
\thispagestyle{empty}
\vspace{-7mm} 
\section*{Abstract}
\begin{quote}
{\footnotesize
We study the soliton flow on the domain of a twistorial harmonic morphism between Riemannian manifolds of dimensions four and three. 
Assuming real-analyticity, we prove that, for the Gibbons--Hawking construction, any soliton flow is uniquely determined by its 
restriction to any local section of the corresponding harmonic morphism. For the Beltrami fields construction, we identify a contour integral whose 
vanishing characterises the trivial soliton flows.}
\end{quote}

\section{Introduction}

A solution of the Ricci flow which evolves by scaling and diffeomorphism is called a \emph{Ricci soliton}.   Specifically, if $g_t = c_t\psi_t^*g$ satisfies the equation
$$
\frac{\partial g_t}{\partial t} = - 2\RicM (g_t)
$$
 on some time interval $[0, \delta )$, where $c_t$ is a family of positive scalars such that $c_0 = 1$ and $\psi_t$ is a family of diffeomorphisms satisfying $\psi_0 ={\rm id}$\,, 
 then
\begin{equation} \label{e:Ric_s}
\RicM+ag+\tfrac12\Lie_{E\!}g=0\;;
\end{equation}
where $2a = c^{\prime}_0$ and the vector field $E$, called the \emph{soliton flow}, is given at each $x \in M$ by $E_x= \frac{\dif}{\dif\!t}\psi_t(x)\vert_{t = 0}$\,.  Conversely, a 
solution to (\ref{e:Ric_s}) on a Riemannian manifold $(M,g)$ gives a small time solution to the Ricci flow equation, under a completeness assumption.   
Ricci solitons occur as rescaled limits at singularity formation and as asymptotic limits of immortal solutions, that is solutions that exist for all future time \cite{Lo}\,.  Both of these limits are interpreted in terms of Cheeger-Gromov-Hamilton pointed convergence of Ricci flows \cite{Ha2}\,.  In the case when $E = \grad\!f$ is the gradient of a function, then the soliton is said to be of \emph{gradient type}.  A soliton is called \emph{shrinking, steady} or \emph{expanding} according as the constant $a$ is negative, zero or positive, respectively.  See the reference \cite{Ch-etal} for an overview.\\  
\indent 
Any soliton metric which is Einstein is called \emph{trivial}.  In dimension $3$\,, any compact soliton is trivial \cite{Iv-1}\,.  On the other hand, in dimension $4$\,, 
non-trivial soliton metrics in the form of K\"ahler metrics of cohomogeneity one on certain projective bundles on $\C\!P^n$ have been found by Koiso \cite{Ko}\,. 
Also, on $\R^n$ there exist the well-known \emph{Gaussian solitons} with flow given by $E=\grad\bigl(-\frac{a}{2}|x-x_0|^2\bigr)$\,, where $x_0\in \R^n$ is some arbitrary point.\\    
\indent 
In this article we are particularly concerned with questions of existence and uniqueness in dimension $4$\,.  In order to address these issues, we make the additional assumption that the metric $g$ supports a twistorial harmonic morphism onto a $3$-manifold (see below).  Up to homotheties, any such harmonic morphism $\phi:(M,g)\to(N,h)$\,, with nonintegrable horizontal distribution,  is locally given either by the so-called Gibbons--Hawking or by the Beltrami fields constructions.  
Both of these methods for constructing a metric involve specifying data on the codomain $N$ from which the metric $g$ is derived.  
A similar idea was applied in \cite{BaiDan} to construct $3$-dimensional solitons, by supposing the existence of a semi-conformal mapping onto a surface.   
Note that by a result in \cite{BaiPan}\,, any horizontally conformal conformal submersion from a $3$-dimensional conformal manifold to a surface $P$ 
can be extended to a unique twistorial map from its $4$-dimensional heaven space to $P$.\\ 
\indent 
In Section \ref{section:GibHaw}\,, we suppose that $(M,g)$ is given by applying the Gibbons-Hawking construction to a real-analytic Riemannian $3$-manifold $(N,h)$\,, 
in particular, $g$ is expressed in terms of a harmonic function on $(N,h)$\,.  As a consequence, this defines a twistorial harmonic morphism $\phi :(M,g) \to (N,h)$\,.  
Then we show that any real-analytic soliton flow on $(M,g)$ is uniquely determined by its restriction to any local section of $\phi$ (Theorem \ref{thm:Ric_s_thm_Gib-Haw}\,).   
We, also, obtain an ansatz for the construction of a soliton flow from a harmonic function and a solution to the monopole equation on $N$ 
(Theorem \ref{thm:charact_w}\,).\\   
\indent 
In Section \ref{sec:Beltrami}\,, we suppose that $(M,g)$ is given by the Beltrami fields construction.  Now, the metric $g$ is given in terms of a $1$-form on $(N, h)$ satisfying the Beltrami fields equation.  Theorem \ref{thm:Ric_s_thm_Beltrami} shows the equivalence between  the triviality of any real-analytic soliton flow $E$ and the vanishing of a contour integral involving the complexification $(\Lie_E\!h)^{\C\!}$ to a local complexification of $N$.

\section{Soliton flows on the domain of a twistorial harmonic morphism}

\indent
A \emph{harmonic morphism} between Riemannian manifolds is a map which, locally, pulls back harmonic functions
to harmonic functions (see \cite{BaiWoo2} for more information on harmonic morphisms).\\
\indent
Let $\phi:(M,g)\to(N,h)$ be a submersive harmonic morphism, with $M$ and $N$ oriented,
$\dim M=4$\,, $\dim N=3$\,. Denote by $\l$ the dilation of $\phi$ and let $\V={\rm ker}\dif\!\phi$ and $\H=\V^{\perp}$
be the vertical and horizontal distributions of $\phi$\,, respectively. We orient $\V$ and $\H$ such that the isomorphisms
$TM=\V\oplus\H$ and $\H=\phi^*(TN)$ be orientation preserving.\\
\indent
Let $V$ be the (fundamental) vertical vector field, characterised by the fact that it is vertical, positive
and $g(V,V)=\l^2$. Then, locally, $g=\l^{-2}\,\phi^*(h)+\l^2\,\theta^2$, where $\theta$ is the vertical dual of $V$;
that is, $\theta(V)=1$ and $\theta|_{\H}=0$\,.\\
\indent
Then $\phi$ is twistorial (with respect to the opposite orientation on $M$) if and only if the following relation holds \cite{PanWoo-sd}\,:
\begin{equation} \label{e:-twist}
\dH\!\bigl(\l^{-2}\bigr)=*_{\H}\O\;,
\end{equation}
where $\O=\dif\!\theta$\,. It follows that, at least outside the set where $\O=0$\,, we have
$c=V(\l^{-2})$ is constant (on each component) and the Ricci tensors $\RicM$ and $\RicN$ of $(M,g)$ and $(N,h)$\,,
respectively, are given by the following relations \cite{PanWoo-exm}\,:
\begin{equation} \label{e:ricci*}
\begin{split}
&\RicM|_{\V}=0\;,\quad\RicM|_{\V\otimes\H}=0\;,\\
&\RicM|_{\H}=\phi^*\bigl(\RicN\bigr)-\frac{c^2}{2}\,\phi^*(h)\;.
\end{split}
\end{equation}
\indent
Furthermore, we have the following facts (see \cite{PanWoo-exm}\,, \cite{PanWoo-sd} and the references therein):\\
\indent
\quad$\bullet$\, $(N,h)$ has constant sectional curvature if and only if $(M,g)$ is self-dual;\\
\indent
\quad$\bullet$\, $(N,h)$ has constant sectional curvature equal to $\frac{c^2}{4}$ if and only if $(M,g)$ is Einstein;
moreover, if $(M,g)$ is Einstein then it is Ricci-flat self-dual.\\
\indent
{}From now on, we shall assume $\O_x\neq0$\,, at each $x\in M$ (if $\O=0$ then, locally, $g$ is a warped-product;
see \cite{Bai-GT} and the references therein for results on soliton flows on such metrics).
Then there exists a unique basic vector field $Z$ such that $\iota_Z\O=0$ and which is projected onto a unit vector field on $(N,h)$\,.

\begin{prop} \label{prop:Ric_s_thm}
Let $\phi:(M,g)\to(N,h)$ be a twistorial harmonic morphism, and let $E$ be a vector field on $M$;
denote by $f$ and $F$ the function and the horizontal vector field, respectively, such that $E=fV+F$.\\
\indent
Then $E$ is a soliton flow, with the corresponding constant $a$\,, if and only if, locally, the following five relations hold:
\begin{equation} \label{e:Ric_s_thm}
\begin{split}
&V(f)+a-\tfrac12\l^2\bigl(fc+z\,\O(X,Y)\bigr)=0\;,\\
&X(f)-y\,\O(X,Y)+\l^{-4}V(x)=0\;,\\
&Y(f)+x\,\O(X,Y)+\l^{-4}V(y)=0\;,\\
&Z(f)+\l^{-4}V(z)=0\;,\\
&\Lie_F\bigl(\phi^*(h)\bigr)+\l^2\bigl(-c^2+fc+z\,\O(X,Y)+2a\l^{-2}\bigr)\phi^*(h)\\
&\hspace{6.3cm}+2\l^2\phi^*\bigl(\RicN\bigr)=0\,\;{\rm on}\;\H\,,
\end{split}
\end{equation}
where $(X,Y,Z)$ is a horizontal frame, projected by $\phi$ onto a positive orthonormal frame on $(N,h)$\,,
and $x\,,\,y\,,\,z$ are functions characterised by $F=xX+yY+zZ$.
\end{prop}
\begin{proof}
Firstly, note that \eqref{e:Ric_s_thm} does not dependent of the pair $(X,Y)$\,, having the stated properties.
Also, \eqref{e:-twist} is (locally) equivalent to the conditions $X(\l)=Y(\l)=0$ and $Z(\l^{-2})=\O(X,Y)$\,.
Thus, by also using $V(\l^{-2})=c$\,, we obtain
\begin{equation} \label{e:for_(1)-(5)}
E(\ln\l)=-\tfrac12\,fc\l^2-\tfrac12\,z\l^2\O(X,Y)\;,
\end{equation}
which implies
\begin{equation} \label{e:1_for_(1)}
E(\ln\l)+V(f)+a=V(f)+a-\tfrac12\l^2\bigl(fc+z\,\O(X,Y)\bigr)\;.
\end{equation}
\indent
Also, we have
\begin{equation} \label{e:2_for_(1)}
\begin{split}
(\Lie_{E\!}g)(V,V)&=(\Lie_{fV\!}g)(V,V)+(\Lie_{F\!}g)(V,V)\\
&=fV(\l^2)+2g([V,fV],V)+F(\l^2)+2g([V,F],V)\\
&=fV(\l^2)+F(\l^2)+2\,V(f)\l^2\\
&=E(\l^2)+2\l^2V(f)\\
&=2\l^2\bigl(E(\ln\l)+V(f)\bigr)\;.
\end{split}
\end{equation}
\indent
Now, \eqref{e:ricci*}\,, \eqref{e:1_for_(1)} and \eqref{e:2_for_(1)} show that the first relation of \eqref{e:Ric_s_thm} is
equivalent to \eqref{e:Ric_s}\,, restricted to $\V$.\\
\indent
Further, for any horizontal vector field $S$, we have
\begin{equation} \label{e:for_(2)-(4)}
\begin{split}
(\Lie_{E\!}g)(V,S)&=g([V,E],S)+g(V,[S,E])\\
&=g\bigl(V(f)\,V+[V,F],S\bigr)+g\bigl(V,S(f)V+[S,F]\bigr)\\
&=g([V,F],S)+\l^2S(f)-\l^2\,\O(S,F)\\
&=\l^2\bigl(S(f)-\O(S,F)+\l^{-4}\phi^*(h)([V,F],S)\bigr)\;.
\end{split}
\end{equation}
\indent
{}From \eqref{e:for_(2)-(4)} we deduce that the second, third and fourth relations of
\eqref{e:Ric_s_thm} are equivalent to \eqref{e:Ric_s}\,, restricted to $\V\otimes\H$.\\
\indent
Next, we have
\begin{equation} \label{e:1_for_(5)}
\begin{split}
(\Lie_{E\!}g)|_{\H}&=\bigl(\Lie_E(\l^{-2}\phi^*(h))\bigr)|_{\H}\\
&=E(\l^{-2})\phi^*(h)|_{\H}+\l^{-2}\bigl(\Lie_E\phi^*(h)\bigr)|_{\H}\\
&=E(\l^{-2})\phi^*(h)|_{\H}+\l^{-2}(\Lie_{fV\!}\phi^*(h))|_{\H}+\l^{-2}(\Lie_{F\!}\phi^*(h))|_{\H}\\
&=E(\l^{-2})\phi^*(h)|_{\H}+\l^{-2}(\Lie_{F\!}\phi^*(h))|_{\H}\\
&=-2\l^{-2}E(\ln\l)\phi^*(h)|_{\H}+\l^{-2}\bigl(\Lie_{F\!}\phi^*(h)\bigr)|_{\H}\;.
\end{split}
\end{equation}
\indent
{}From \eqref{e:for_(1)-(5)} and \eqref{e:1_for_(5)} we obtain
\begin{equation} \label{e:2_for_(5)}
(\Lie_{E\!}g)|_{\H}=\bigl(fc+z\,\O(X,Y)\bigr)\phi^*(h)|_{\H}+\l^{-2}\bigl(\Lie_{F\!}\phi^*(h)\bigr)|_{\H}\;.
\end{equation}
\indent
{}From \eqref{e:ricci*} and \eqref{e:2_for_(5)} we deduce that the fifth relation of
\eqref{e:Ric_s_thm} is equivalent to \eqref{e:Ric_s}\,, restricted to $\H$.
\end{proof}

\section{Ricci solitons and the Gibbons--Hawking construction} \label{section:GibHaw}

\indent
Let $\phi:(M,g)\to(N,h)$ be a twistorial harmonic morphism with $V(\l^{-2})=0$\,. This is equivalent to the fact that, locally,
there exists a function $u$ and a one-form $A$ on $N$ satisfying the monopole equation $\dif\!u=*\dif\!A$
and such that $$g=u\,h+u^{-1}(\dif\!t+A)^2\;,$$ with $\phi:M=N\times\R\to N$ the projection;
in particular, $\l^{-2}=u$, $V=\frac{\partial}{\partial t}$ is a Killing vector field,
and $\theta=\dif\!t+A$\,.\\
\indent
Note that, if $(N,h)$ is real-analytic then, as $u$ and $A$ satisfy $\D u=0$ and $\D A=0$\,,
we have that, also, $g$ is real-analytic.\\
\indent
Next, we state the main result of this section.

\begin{thm} \label{thm:Ric_s_thm_Gib-Haw}
Let $(M,g)$ be given by the Gibbons--Hawking construction, with $(N,h)$ real-analytic.\\
\indent
Then any real-analytic soliton flow on $(M,g)$ is uniquely determined by its restriction to
any local section of $\phi$\,.
\end{thm}

\indent
The proof of Theorem \ref{thm:Ric_s_thm_Gib-Haw} will be given below.\\
\indent
From now on, in this section, \emph{we shall denote by $X$, $Y$ and $Z$ the projections onto $N$ of the corresponding
vector fields appearing in \eqref{e:Ric_s_thm}\,.} Then, as $\O=\dif\!A$\,, we have $Z=\frac{1}{|\!\dif\!u|}\,\grad u$\,, $X(u)=Y(u)=0$\,.
Hence, the horizontal lift of $X$ is $\widetilde{X}=-A(X)\frac{\partial}{\partial t}+X$ and, similarly, for $Y$ and $Z$\,.
Consequently, $\O(X,Y)=Z(u)=|\!\dif\!u|$\,.\\
\indent
Thus, Proposition \ref{prop:Ric_s_thm} gives the following result.

\begin{cor} \label{cor:Ric_s_thm_Gib-Haw}
Let $(M,g)$ be given by the Gibbons--Hawking construction. Then a vector field $E=fV+F$, with $F$ horizontal,
is a soliton flow on $(M,g)$\,, with the corresponding constant $a$\,, if and only if the following five relations hold:\\
\begin{equation} \label{e:Ric_s_thm_Gib-Haw}
\begin{split}
&\frac{\partial f}{\partial t}+a-\frac12\,zu^{-1}|\!\dif\!u|=0\;,\\
&-A(X)\frac{\partial f}{\partial t}+X(f)-y\,|\!\dif\!u|+u^2\,\frac{\partial x}{\partial t}=0\;,\\
&-A(Y)\frac{\partial f}{\partial t}+Y(f)+x\,|\!\dif\!u|+u^2\,\frac{\partial y}{\partial t}=0\;,\\
&-A(Z)\frac{\partial f}{\partial t}+Z(f)+u^2\,\frac{\partial z}{\partial t}=0\;,\\
&\Lie_{F\!}h+2u^{-1}\RicN+u^{-1}\bigl(z\,|\!\dif\!u|+2au\bigr)h=0\,\;{\rm on}\;\H\,.
\end{split}
\end{equation}
\end{cor}

\indent
Assuming real-analyticity, Corollary \ref{cor:Ric_s_thm_Gib-Haw} gives the following result.

\begin{cor} \label{cor:Ric_s_Gib-Haw}
Let $(M,g)$ be given by the Gibbons--Hawking construction, with $(N,h)$ real-analytic.
Let $E=fV+F$ be a real-analytic vector field on $(M,g)$\,, where $F$ is horizontal.\\
\indent
On writing, locally, $f=\sum_{j=0}^{\infty}t^jf_j$ and $F=\sum_{j=0}^{\infty}t^j\widetilde{F_j}$\,,
with $f_j$ functions on $N$, and $\widetilde{F_j}$ the horizontal lifts of vector fields $F_j=x_jX+y_jY+z_jZ$ on $N$,
we have that $E$ is a soliton flow on $(M,g)$\,, with the corresponding constant $a$\,, if and only if the following relations hold:\\
\begin{equation} \label{e:Ric_s_Gib-Haw}
\begin{split}
&f_{j+1}=\frac{1}{2(j+1)}\,z_ju^{-1}|\!\dif\!u|\;,\;\bigl(j\in\mathbb{N}\setminus\{0\}\bigr)\;,\\
&f_1=\frac12\,z_0u^{-1}|\!\dif\!u|-a\;,\\
&-(j+1)A(X)f_{j+1}+X(f_j)-y_j|\!\dif\!u|+(j+1)u^2x_{j+1}=0\;,\;\bigl(j\in\mathbb{N}\bigr)\;,\\
&-(j+1)A(Y)f_{j+1}+Y(f_j)+x_j|\!\dif\!u|+(j+1)u^2y_{j+1}=0\;,\;\bigl(j\in\mathbb{N}\bigr)\;,\\
&-(j+1)A(Z)f_{j+1}+Z(f_j)+(j+1)u^2z_{j+1}=0\;,\;\bigl(j\in\mathbb{N}\bigr)\;,\\
&\Lie_{F_j\!}h=2(j+1)A\odot F_{j+1}^{\flat}-z_ju^{-1}|\!\dif\!u|\,h\;,\;\bigl(j\in\mathbb{N}\setminus\{0\}\bigr)\;,\\
&\Lie_{F_0\!}h+2u^{-1}\RicN+\bigl(2a+z_0u^{-1}|\!\dif\!u|\bigr)h-2A\odot F_1^{\flat}=0\;.
\end{split}
\end{equation}
\end{cor}
\begin{proof}
It is easy to see that the first four relations of \eqref{e:Ric_s_thm_Gib-Haw} are equivalent to the first five relations
of \eqref{e:Ric_s_Gib-Haw}\,.\\
\indent
Let $S$ be a vector field on $N$ and let $\widetilde{S}$ be its horizontal lift.
As $\widetilde{S}=-A(S)\frac{\partial}{\partial t}+S$, we have
$\widetilde{S}(t^j)=-jA(S)t^{j-1}$. Hence,
$$(\Lie_{t^j\widetilde{F_j}}h)(\widetilde{S},\widetilde{S})=t^j(\Lie_{F_j\!}h)(S,S)-2jt^{j-1}\bigl(A\odot F_j^{\flat}\bigr)(S,S)\;,$$
from which the last two relations follow quickly.
\end{proof}

\indent
Now, we can give the proof of Theorem \ref{thm:Ric_s_thm_Gib-Haw}\,.

\begin{proof}[Proof of Theorem \ref{thm:Ric_s_thm_Gib-Haw}]
It is sufficient to prove that $E$ is determined by its restriction to $N\times\{0\}$\,.\\
\indent
The first and the fifth equations of \eqref{e:Ric_s_Gib-Haw} give
$$-\frac12\,z_ju^{-1}|\!\dif\!u|A(Z)+\frac{1}{2j}\,Z\bigl(z_{j-1}u^{-1}|\!\dif\!u|\bigr)+(j+1)u^2z_{j+1}=0\;,$$
for any $j\in\mathbb{N}\setminus\{0\}$\,.\\
\indent
Also, the second and the fifth relations of \eqref{e:Ric_s_Gib-Haw} give
\begin{equation} \label{e:f_0}
-\frac12\,z_0u^{-1}|\!\dif\!u|A(Z)+aA(Z)+Z(f_0)+u^2z_1=0\;.
\end{equation}
\indent
Thus, $z_j$ and $f_j$ are determined by $z_0$ and $f_0$\,, for any $j\in\mathbb{N}$\,.\\
\indent
To complete the proof, just note that the third and fourth relations of \eqref{e:Ric_s_Gib-Haw} imply that
$x_j$ and $y_j$ are determined by $x_0$\,, $y_0$\,, $z_0$ and $f_0$\,, for any $j\in\mathbb{N}$\,.
\end{proof}

\indent
We, also, obtain that if $E$ is a real analytic
soliton flow on a Riemannian manifold, given by the Gibbons--Hawking construction, then its vertical part
is determined by its horizontal part, up to the choice of $f_0$ satisfying \eqref{e:f_0}\,.

\section{A particular case}

\indent
In this section, we continue the study of Section \ref{section:GibHaw} by tackling a particular case.

\begin{cor} \label{cor:Ric_s_F0=0=F2}
Let $(M,g)$ be given by the Gibbons--Hawking construction, with $(N,h)$ real-analytic.
Let $E=fV+t\widetilde{F_1}$ be a real-analytic vector field on $(M,g)$\,, where $\widetilde{F_1}$ is the horizontal lift
of the vector field $F_1$ on $N$.\\
\indent
Then $E$ is a soliton flow on $(M,g)$\,, with the corresponding constant $a$\,, if and only if
$f=f_0-at+\frac14\,b\,t^2$\,, where $b$ is a constant, $f_0$ is a function on $N$, and the following
relations hold:\\
\begin{equation} \label{e:Ric_s_F0=0=F2}
\begin{split}
&b=u^{-1}F_1(u)\;,\\
&bA=2\dif\!u\times F_1^{\flat}\;,\\
&aA+\dif\!f_0+u^2F_1^{\flat}=0\;,\\
&\Lie_{F_1\!}h=-\,b\,h\;,\\
&\RicN+\,a\,u\,h-uA\odot F_1^{\flat}=0\;.
\end{split}
\end{equation}
\end{cor}
\begin{proof}
{}From the first two relations of \eqref{e:Ric_s_Gib-Haw} we obtain that $f_j=0$\,, for any $j\geq3$\,,
$f_2=\frac14\,z_1u^{-1}|\!\dif\!u|$\,, and $f_1=-a$\,.\\
\indent
Note that the third, fourth and fifth relations of \eqref{e:Ric_s_Gib-Haw} are equivalent to
\begin{equation} \label{e:Ric_s_Gib-Haw_345}
-(j+1)f_{j+1}A+\dif\!f_j+\dif\!u\times F_j^{\flat}+(j+1)u^2F_{j+1}^{\flat}=0\;,
\end{equation}
for any $j\in\mathbb{N}$\,.\\
\indent
Now, if $j\geq3$ then \eqref{e:Ric_s_Gib-Haw_345} is trivially satisfied, whilst for $j=2$ it gives that $\dif\!f_2=0$\,.
Thus, $b=z_1u^{-1}|\!\dif\!u|$ is constant; equivalently, $b=u^{-1}F_1(u)$ is constant.\\
\indent
The second and third relations of \eqref{e:Ric_s_F0=0=F2} are equivalent to \eqref{e:Ric_s_Gib-Haw_345} with $j=1$ and
$j=0$\,, respectively.\\
\indent
Finally, the last two relations of \eqref{e:Ric_s_F0=0=F2} are equivalent to the last two relations of \eqref{e:Ric_s_Gib-Haw}\,.
\end{proof}

\indent
With the same notations as in Corollary \ref{cor:Ric_s_Gib-Haw}\,, it is easy to see that we can weaken
the hypotheses of Corollary \ref{cor:Ric_s_F0=0=F2} to $(N,h)$ real-analytic and $F_0=F_2=0$\,.\\
\indent
Also, the second relations of \eqref{e:Ric_s_F0=0=F2} imply $A(\grad u)=0$ (condition which can always be satisfied,
locally). Moreover, the first two relations of \eqref{e:Ric_s_F0=0=F2} determine $F_1$\,, whilst the third requires
$\dif\bigl(aA+u^2F_1^{\flat}\bigr)=0$ to determine $f_0$\,, locally, up to a constant. Consequently, $b=0$ if and only if $F_1=0$
which implies $a=0$\,, $f_0$ is constant (otherwise, $\dif\!A=0$), and $(M,g)$ is Ricci flat self-dual.\\
\indent
Therefore from now on we shall assume $b\neq0$\,. 

\begin{lem} \label{lem:Ric_s_F0=0=F2}
Let $u$ be a harmonic function on a three-dimensional Riemannian manifold $(N,h)$\,.
Let $B$ be a local solution of the monopole equation $\dif\!u=*\dif\!B$ and let $w$ be a function
satisfying $(\grad w)(u)=bu^3+aB^{\sharp}(u)$\,.\\
\indent
Then there exists functions $v$ and $f_0$ such that the first three relations of \eqref{e:Ric_s_F0=0=F2}
are satisfied, with $A=B+\dif\!v$ and $F_1=-au^{-2}B^{\sharp}+u^{-2}\grad w$\,, if and only if
\begin{equation} \label{e:lem-Ric_s_F0=0=F2}
b\dif\!v-2u^{-2}\dif\!u\times\dif\!w+bB+2au^{-2}\dif\!u\times B=0\;.
\end{equation}
\indent
Consequently, the first three equations of \eqref{e:Ric_s_F0=0=F2} can be, locally, solved, up to a gauge transformation,
if and only if there exists a function $w$ such that the following two assertions hold:\\
\indent
\quad{\rm (i)} $(\grad u)(w)=bu^3+aB^{\sharp}(u)$\,;\\
\indent
\quad{\rm (ii)} $-2u^{-2}\dif\!u\times\dif\!w+bB+2au^{-2}\dif\!u\times B$ is closed.
\end{lem}
\begin{proof}
As $A$ satisfies $\dif\!u=*\dif\!A$\,, it must (locally) be of the form $A=B+\dif\!v$\,, for a suitable function $v$.\\
\indent
Also, the third relation of \eqref{e:Ric_s_F0=0=F2} is satisfied, for a suitable $f_0$\,, if and only if
$a\dif\!u=-*\dif\bigl(u^2F_1^{\flat}\bigr)$\,. Thus, $F_1=-au^{-2}B^{\sharp}+u^{-2}\grad w$ for a suitable function $w$.
Then the first relation of \eqref{e:Ric_s_F0=0=F2} is satisfied if and only if $\grad w=bu^3+aB(\grad u)$\,.\\
\indent
To complete the proof, just note that, now, the second relation of \eqref{e:Ric_s_F0=0=F2}
is equivalent to \eqref{e:lem-Ric_s_F0=0=F2}\,.
\end{proof}

\indent
We do not have a reference for the following simple and easy to prove lemma.

\begin{lem} \label{lem:codif-wedge}
Let $\a$ and $\b$ be two one-forms on a Riemannian manifold, and let $\dif^*$ be the codifferential.
Then $$\dif^{*\!}(\a\wedge\b)=(\dif^*\!\a)\,\b-\a\,(\dif^*\!\b)-\bigl[\a^{\sharp},\b^{\sharp}\bigr]^{\flat}\;.$$
\end{lem}

\indent
Next, we use Lemma \ref{lem:codif-wedge} to characterise the suitable functions $w$ of Lemma \ref{lem:Ric_s_F0=0=F2}\,.

\begin{prop} \label{prop:first_3_of_Ric_s_F0=0=F2}
The first three equations of \eqref{e:Ric_s_F0=0=F2} can be, locally, solved, up to a gauge transformation,
if and only if there exists a function $w$ on $(N,h)$ such that the following two equations hold:\\
\begin{equation} \label{e:first_3_of_Ric_s_F0=0=F2}
\begin{split}
(\grad u)(w)=bu^3+aB^{\sharp}(&u)\,,\\
\bigl(2u^{-2}\D w-2au^{-2}\dif^*\!B+b\bigr)\dif\!u+2\bigl[u^{-2}\grad&\,u,\grad w-aB^{\sharp}\bigr]^{\flat}\\
&-4u^{-3}|\!\dif\!u|^2(\dif\!w-aB)=0\;.
\end{split}
\end{equation}
\end{prop}
\begin{proof}
We have to show that the second relation of \eqref{e:first_3_of_Ric_s_F0=0=F2} holds if and only if
the one-form appearing in (ii) of Lemma \ref{lem:Ric_s_F0=0=F2} is closed; equivalently,
\begin{equation} \label{e:for_first_3_of_Ric_s_F0=0=F2}
*\dif\bigl(-2u^{-2}*(\dif\!u\wedge\dif\!w)+2au^{-2}*(\dif\!u\wedge B)+bB\bigr)=0\;.
\end{equation}
\indent
We calculate the left hand side of \eqref{e:for_first_3_of_Ric_s_F0=0=F2} by using that $*\dif\!B=\dif\!u$
and Lemma \ref{lem:codif-wedge}\,:
\begin{equation}
\begin{split}
-2\dif^*\!\bigl((u^{-2}\dif\!u)\wedge(\dif\!w-aB)\bigr)\,&+b\dif\!u\\
=\,-2\bigl(\dif^*(u^{-2}\dif\!u)\bigr)(\dif\!w-aB)\,&+2(\dif^*\!\dif\!w-a\dif^*\!B)\,u^{-2}\dif\!u\\
+2\bigl[u^{-2}&\grad u,\grad w-aB^{\sharp}\bigr]+b\dif\!u\;.
\end{split}
\end{equation}
\indent
Now, just note that, as $u$ is harmonic, we have $\dif^*(u^{-2}\dif\!u)=2u^{-3}|\!\dif\!u|^2$.
\end{proof}

\indent
Note that, the first two relations of \eqref{e:charact_w}\,, below, are tensorial in $\grad w$\,.

\begin{thm} \label{thm:charact_w}
Let $u$ be a harmonic function on a three-dimensional real-analytic Riemannian manifold $(N,h)$\,.
Let $B$ be a (local) solution of the monopole equation $\dif\!u=*\dif\!B$ and let $w$ be a function on $N$ satisfying
\begin{equation} \label{e:charact_w}
\begin{split}
(\grad w)(u)&=bu^3+aB^{\sharp}(u)\;,\\
\nabla_{\grad w}\dif\!u&=-u^{-1}|\!\dif\!u|^2\bigl(\dif\!w-aB\bigr)+\tfrac52\,bu^2\dif\!u-a\bigl[\grad u,B^{\sharp}\bigr]^{\flat}
+\tfrac12\,a\bigl(\Lie_{B^{\sharp}\!}h\bigr)(\grad u,\cdot)\;,\\
\nabla\!\dif\!w&=\tfrac12\,a\Lie_{B^{\sharp}\!}h-\tfrac12\,bu^2h+2u^{-1}\dif\!u\odot\bigl(\dif\!w-aB\bigr)\;,
\end{split}
\end{equation}
where $a,b\in\R$\,, $b\neq0$\,.\\
\indent
Then, locally on $N$, there exist functions $v$ and $f_0$\,, unique up to constants, such that
if $(M,g)$ is given by $(N,h)$\,, $u$ and $A=B+\dif\!v$\,, through the Gibbons--Hawking construction,
and $E=\bigl(f_0-at+\frac14\,b\,t^2\bigr)\frac{\partial}{\partial\,t}+tF$, where $F$ is the horizontal lift of
$u^{-2}\bigl(\grad w-aB^{\sharp}\bigr)$\,, then the following assertions are equivelent:\\
\indent
\quad{\rm (i)} $E$ is a soliton flow on $(M,g)$\,, with the corresponding constant $a$\,;\\
\indent
\quad{\rm (ii)} $\RicN+\,a\,u\,h-u^{-1}A\odot(\dif\!w-aB)=0$\,.\\
\indent
Moreover, any soliton flow on a real-analytic Riemannian manifold, given by the Gibbons--Hawking construction (with $u$ nonconstant), 
is obtained this way, if its horizontal part is $t$ times a nonzero basic vector field.
\end{thm}
\begin{proof}
If we denote $G=\grad w-aB^{\sharp}$ then $a\dif\!u=-*\dif\!G^{\flat}$\,,
and \eqref{e:first_3_of_Ric_s_F0=0=F2} is equivalent to
\begin{equation} \label{e:for_first_3_of_Ric_s_F0=0=F2_bis}
\begin{split}
G(u)&=bu^3\\
2\bigl[u^{-2}\grad u,G\bigr]&=-\bigl(2u^{-2}\dif^*\!G+b\bigr)\grad u+4u^{-3}|\!\dif\!u|^2G\;.
\end{split}
\end{equation}
\indent
Further, as $F_1=F=u^{-2}G$\,, the fourth relation of \eqref{e:Ric_s_F0=0=F2} holds if and only if
\begin{equation} \label{e:fourth_of_Ric_s_F0=0=F2_with_G}
\Lie_{G\!}h=-bu^2h+4u^{-1}\dif\!u\odot G^{\flat}\;.
\end{equation}
\indent
Then \eqref{e:fourth_of_Ric_s_F0=0=F2_with_G} and the first relation of \eqref{e:for_first_3_of_Ric_s_F0=0=F2_bis}
imply $\dif\!^*G=-\frac{bu^2}{2}$ (which, together with the second relation of \eqref{e:for_first_3_of_Ric_s_F0=0=F2_bis}\,,
gives $G\bigl(|\!\dif\!u|^2\bigr)=3\,b\,u^2|\!\dif\!u|^2$\,).\\
\indent
Thus, if \eqref{e:fourth_of_Ric_s_F0=0=F2_with_G} holds then \eqref{e:for_first_3_of_Ric_s_F0=0=F2_bis} is equivalent to
the first relation of \eqref{e:charact_w} and the following
$\bigl[\grad u,G\bigr]=2u^{-1}|\!\dif\!u|^2G-2bu^2\grad u\,;$
further, the latter is equivalent to
\begin{equation} \label{e:charact_w_1}
\nabla_{\grad u}(\grad w)=\nabla_{\grad w}(\grad u)+a\bigl[\grad u,B^{\sharp}\bigr]+2u^{-1}|\!\dif\!u|^2\bigl(\grad w-aB^{\sharp}\bigr)
-2bu^2\grad u\;.
\end{equation}
\indent
On the other hand, \eqref{e:fourth_of_Ric_s_F0=0=F2_with_G} is equivalent to
\begin{equation} \label{e:charact_w_2}
\nabla\!\dif\!w=\tfrac12\,a\Lie_{B^{\sharp}\!}h-\tfrac12\,bu^2h+2u^{-1}\dif\!u\odot\bigl(\dif\!w-aB\bigr)\;;
\end{equation}
which, together with $G(u)=bu^3$, implies
\begin{equation} \label{e:charact_w_3}
\nabla_{\grad u}(\dif\!w)=\tfrac12\,a\bigl(\Lie_{B^{\sharp}\!}h\bigr)\bigl(\grad u,\cdot\bigr)
+\tfrac12\,bu^2\dif\!u+u^{-1}|\!\dif\!u|^2\bigl(\dif\!w-aB\bigr)\;.
\end{equation}
\indent
Thus, if \eqref{e:charact_w_2} holds then \eqref{e:charact_w_1} is equivalent to
the second relation of \eqref{e:charact_w}\,.\\
\indent
We have, thus, shown that the first four relations of \eqref{e:Ric_s_F0=0=F2} can be locally solved, with a suitable $A=B+\dif\!v$ and $f_0$\,,
if and only if \eqref{e:charact_w} holds. Together with Corollary \ref{cor:Ric_s_F0=0=F2}\,, this completes the proof.
\end{proof}

\indent 
We end this section with the following application of Theorem \ref{thm:charact_w}\,. We omit the proof. 

\begin{cor} \label{cor:Ric_s_F0=0=F2_u_particular}
Let $(M,g)$ be given by the Gibbons--Hawking construction, with $(N,h)$ real-analytic, and the fibres of $u$ are flat and geodesic.\\ 
\indent 
Let $E=fV+t\widetilde{F}$ be a real-analytic vector field on $(M,g)$\,, where $\widetilde{F}$ is the horizontal lift
of the vector field $F$ on $N$.\\
\indent
Then $E$ is a soliton flow on $(M,g)$\,, with the corresponding constant $a$\,, if and only if $F=0$\,, 
$a=0$\,, $f_0$ is constant, and $(M,g)$ is Ricci flat self-dual. 
\end{cor}

\section{Ricci solitons and the Beltrami fields construction} \label{sec:Beltrami}

\indent
Let $\phi:(M,g)\to(N,h)$ be a twistorial harmonic morphism with $V(\l^{-2})\neq0$\,. Then, up to homotheties, we
may suppose that $V(\l^{-2})=2$\,. It follows that, locally, there exists a one-form $A$ on $N$
satisfiying the Beltrami fields equation $\dif\!A+2*A=0$ and such that $$g=\r^2h+\r^{-2}(\r\dif\!\r+A)^2\;,$$
with $\phi:M=N\times(0,\infty)\to N$ the projection; in particular, $\l=\r^{-1}$, $V=\r^{-1}\frac{\partial}{\partial\r}$\,,
and $\theta=\r\dif\!\r+A$\,.\\
\indent
Note that, if $(N,h)$ is real-analytic then, as $A$ satisfies $\D A=4A$ (here, $\D$ is the Hodge-Laplace operator),
we have that, also, $g$ is real-analytic. Furthermore, the complexification
of $g$ is defined on $N^{\C}\times(\C\setminus\{0\})$.

\begin{thm} \label{thm:Ric_s_thm_Beltrami}
Let $E$ be a real-analytic soliton flow on a Riemannian manifold $(M,g)$ given by the Beltrami fields construction.\\
\indent
Suppose that $E$ admits a complexification on a set containing $U\times\g$\,, where $U$ is an open subset of $N$
and $\g$ is a circle on $\C\!$, centred at $0$\,.\\
\indent
Then the following assertions are equivalent:\\
\indent
\quad{\rm (i)} $(M,g)$ is Einstein;\\
\indent
\quad{\rm (ii)} $\int_{\g}\r\,(\Lie_E\!h)^{\C\!}\dif\!\r=0$\,, on $U^{\C}$.\\
\indent
Furthermore, a sufficient condition for {\rm (i)} and {\rm (ii)} to hold is that there exists $j\in\mathbb{N}\setminus\{0\}$
such that the trace-free part of $\int_{\g}\r^{2j+1}\,(\Lie_E\!h)^{\C\!}\dif\!\r$ is zero, on $U^{\C}$.
\end{thm}

\indent
The proof of Theorem \ref{thm:Ric_s_thm_Beltrami} will be given below.\\
\indent
From now on, in this section, \emph{we shall denote by $X$, $Y$ and $Z$ the projections onto $N$ of the corresponding
vector fields appearing in \eqref{e:Ric_s_thm}\,.} Then, as $\O=\dif\!A$\,, we have $Z=\frac{1}{|A|}A^{\sharp}$, $A(X)=A(Y)=0$\,.
Hence, the horizontal lifts of $X$, $Y$ and $Z$ are $X$, $Y$ and $-|A|\r^{-1}\frac{\partial}{\partial\r}+Z$, respectively.
Consequently, $\O(X,Y)=\bigl(-|A|\r^{-1}\frac{\partial}{\partial\r}+Z\bigr)(\r^2)=-2|A|$\,.\\
\indent
Thus, Proposition \ref{prop:Ric_s_thm} gives the following result.

\begin{cor} 
Let $(M,g)$ be given by the Beltrami fields construction. Then a vector field $E=fV+F$, with $F$ horizontal,
is a soliton flow on $(M,g)$\,, with the corresponding constant $a$\,, if and only if the following five relations hold:\\
\begin{equation} \label{e:Ric_s_thm_Beltrami}
\begin{split}
&\r\,\frac{\partial f}{\partial\r}+a\r^2-f+z|A|=0\;,\\
&X(f)+2y|A|+\r^3\,\frac{\partial x}{\partial\r}=0\;,\\
&Y(f)-2x|A|+\r^3\,\frac{\partial y}{\partial\r}=0\;,\\
&-|A|\r^{-1}\frac{\partial f}{\partial\r}+Z(f)+\r^3\,\frac{\partial z}{\partial \r}=0\;,\\
&\r^2\Lie_{F\!}h+2\RicN+2\bigl(f-z|A|+a\r^2-2\bigr)h=0\,\;{\rm on}\;\H\,.
\end{split}
\end{equation}
\end{cor}

\indent
We shall, now, rewrite \eqref{e:Ric_s_thm_Beltrami} under the hypothesis of Theorem \ref{thm:Ric_s_thm_Beltrami}\,.

\begin{cor} \label{cor:Ric_s_Beltrami_Laurent}
Let $(M,g)$ be given by the Beltrami fields construction, with $(N,h)$ real-analytic.
Let $E=fV+F$ be a real-analytic vector field on $(M,g)$\,, where $F$ is horizontal.\\
\indent
Suppose that $E$ admits a complexification on a set containg $U\times\g$\,, where $U$ is an open subset of $N$
and $\g$ is a circle on $\C\!$, centred at $0$\,.\\
\indent
Then $f$ and $F$ admit Laurent series expansions
$f=\sum_{j=-\infty}^{\infty}\r^jf_j$ and $F=\sum_{j=-\infty}^{\infty}\r^j\widetilde{F_j}$, where $f_j$ are functions on $N$,
and $\widetilde{F_j}$ are the horizontal lifts of vector fields $F_j=x_jX+y_jY+z_jZ$ on $N$\,.\\
\indent
Furthermore, $E$ is a soliton flow on $(M,g)$\,, with the corresponding constant $a$\,, if and only if the following relations hold:\\
\begin{equation} \label{e:Ric_s_Beltrami_Laurent}
\begin{split}
&(j-1)f_j+z_j|A|=0\;,\;\bigl(j\in\mathbb{Z}\setminus\{2\}\bigr)\;,\\
&f_2+z_2|A|+a=0\;,\\
&X(f_j)+2y_j|A|+(j-2)x_{j-2}=0\;,\;\bigl(j\in\mathbb{Z}\bigr)\;,\\
&Y(f_j)-2x_j|A|+(j-2)y_{j-2}=0\;,\;\bigl(j\in\mathbb{Z}\bigr)\;,\\
&-(j+2)|A|f_{j+2}+Z(f_j)+(j-2)z_{j-2}=0\;,\;\bigl(j\in\mathbb{Z}\bigr)\;,\\
&\Lie_{F_{j-2}}h-2jA\odot F_j^{\flat}+2\bigl(f_j-z_j|A|\bigr)h=0\;,\;\bigl(j\in\mathbb{Z}\setminus\{0\,,\,2\}\bigr)\;,\\
&\Lie_{F_{-2}}h+2\RicN+2\bigl(f_0-z_0|A|-2\bigr)h=0\;,\\
&\Lie_{F_0}h-4A\odot F_2^{\flat}+2\bigl(f_2-z_2|A|+a\bigr)h=0\;,\\
\end{split}
\end{equation}
where $\odot$ denotes he symmetric product.
\end{cor}
\begin{proof}
We shall denote by the same symbol an object and its complexification. Also, all the objects are assumed complex-analytic.\\
\indent
As the domain (of the complexification) of $E$ contains $U$ times a circle, with $U\subseteq N$ open, it, also, contains $U$ times
an open annulus. Then the existence of the Laurent series expansions, for $f$, $x$\,, $y$\,, $z$\,, and, consequently, $F$,
follows by applying a standard argument.\\
\indent
It is easy to see that the first four relations of \eqref{e:Ric_s_thm_Beltrami} are equivalent to the first five relations
of \eqref{e:Ric_s_Beltrami_Laurent}\,.\\
\indent
Let $S$ be a vector field on $N$ and let $\widetilde{S}$ be its horizontal lift. As
$\widetilde{S}=-A(S)\r^{-1}\frac{\partial}{\partial\r}+S$\,, we have $\widetilde{S}(\r^j)=-jA(S)\r^{j-2}$.
Consequently,
$$(\Lie_{\r^j\widetilde{F_j}}h)(\widetilde{S},\widetilde{S})=\r^j(\Lie_{F_j}h)(S,S)-2j\r^{j-2}\bigl(A\odot F_j^{\flat}\bigr)(S,S)\;,$$
from which the last three relations follow quickly.
\end{proof}

\indent
We can, now, give the proof of Theorem \ref{thm:Ric_s_thm_Beltrami}\,.

\begin{proof}[Proof of Theorem \ref{thm:Ric_s_thm_Beltrami}]
The equivalence of assertions (i) and (ii) is a consequence of the seventh formula of \eqref{e:Ric_s_Beltrami_Laurent}\,.\\
\indent
On working with homogeneous quadratic polynomials, instead of symmetric bilinear forms,
and identifying forms and vector fields on $N$, through $h$\,,
the sixth formula of \eqref{e:Ric_s_Beltrami_Laurent} is equivalent to the following:\\
\begin{equation} \label{e:6_of_Ric_s_Beltrami_Laurent}
\begin{split}
-\tfrac12\Lie_{F_{j-2}}h=(f_j-z&_j|A|)X^2+(f_j-z_j|A|)Y^2+(f_j-2z_j|A|)Z^2\\
&-jx_j|A|XZ-jy_j|A|YZ\;.
\end{split}
\end{equation}
\indent
The last assertion follows quickly from the first formula of \eqref{e:Ric_s_Beltrami_Laurent} and
\eqref{e:6_of_Ric_s_Beltrami_Laurent}\,.
\end{proof}

\begin{rem}
Under the same hypotheses as in Theorem \ref{thm:Ric_s_thm_Beltrami}\,, and with similar proofs, the following two statements
follow quickly:\\
\indent
(a) The vertical part of $E$ is determined by the horizontal part (in fact, by $z=A(E)$\,).\\
\indent
(b) $(M,g)$ is self-dual if and only if the trace-free part of $\int_{\g}\r\,(\Lie_E\!h)^{\C\!}\dif\!\r$ is zero, on~$U^{\C}$.\\
\indent
(c) $F_0$ and $F_1$ uniquely determine all $F_j$\,, with $j\in\mathbb{N}$\,.
\end{rem}

\indent
We end with a particular case.

\begin{cor} \label{cor:Ric_s_thm_Beltrami}
Let $(M,g)$ be a Riemannian manifold given by the Beltrami fields construction,
with $A$ the corresponding one-form on $N$.\\
\indent
Suppose that $E$ admits a complexification on a set containg $U\times\g$\,, where $U$ is an open subset of $N$
and $\g$ is a circle on $\C\!$, centred at $0$\,.\\
\indent
Then $E=fV+F$, with $F=xX+yY+zZ$ and $z$ a function on $N$, is a soliton flow, with the corresponding constant $a$\,,
if and only if the following assertions hold:\\
\indent
\quad{\rm (i)} $a=0$\,, $f=z|A|$ and $Z(f)=0$\,;\\
\indent
\quad{\rm (ii)} $x+{\rm i}y=u\,e^{-{\rm i}\r^{-2}|A|}+\frac{1}{2{\rm i}|A|}(X+{\rm i}\,Y)(f)$\,, where $u$
is a function on $N$;\\
\indent
\quad{\rm (iii)} $\int_{\g}\r\,(\Lie_E\!h)^{\C\!}\dif\!\r+2\RicN-4h=0$\,, on $U^{\C}$;\\
\indent
\quad{\rm (iv)} $\int_{\g}\r^{-j+1}(\Lie_E\!h)^{\C\!}\dif\!\r=A\odot\int_{\g}\r^{-j-1}(F^{\flat})^{\C\!}\dif\!\r$\,, on $U^{\C}$,
for any $j\in\mathbb{Z}\setminus\{0\}$\,.
\end{cor}
\begin{proof}
As $\frac{\partial z}{\partial\r}=0$\,, the first assertion of \eqref{e:Ric_s_thm_Beltrami} is equivalent to
$$\frac{\partial}{\partial\r}(\r^{-1}f+a\r-\r^{-1}z|A|)=0\;;$$
equivalently, $f=-a\r^2+v\r+z|A|$\,, where $v$ is a function on $N$.\\
\indent
By using this, we obtain that the fourth relation of \eqref{e:Ric_s_thm_Beltrami} is equivalent to
$$-|A|^{-1}\r^{-1}(v-2a\r)+\r\,Z(v)+Z(z|A|)=0\;;$$
that is, $a=v=Z(v)=0$\,.\\
\indent
Then the second and the third relations of \eqref{e:Ric_s_thm_Beltrami} are equivalent to assertion (ii)\,.
Also, by using Corollary \ref{cor:Ric_s_Beltrami_Laurent}\,, we obtain that the fifth relation of \eqref{e:Ric_s_thm_Beltrami}
is equivalent to (iii) and (iv)\,.
\end{proof}

\end{document}